\documentclass[fleqn,12pt,twoside]{article}

\usepackage[headings]{espcrc1}
\readRCS
$Id: espcrc1.tex,v 1.2 2004/02/24 11:22:11 spepping Exp $
\usepackage{graphicx}
\usepackage[figuresright]{rotating}

\newtheorem{theorem}{Theorem}

\newtheorem{conjecture}[theorem]{Conjecture}
\newtheorem{corollary}[theorem]{Corollary}

\newtheorem{problem}{Problem}
\newtheorem{proposition}[theorem]{Proposition}

\newenvironment{proof}[1][Proof.]{\begin{trivlist}
\item[\hskip \labelsep {\bfseries #1}]}{\end{trivlist}}

\newenvironment{acknowledgement}[1][Acknowledgement]{\begin{trivlist}
\item[\hskip \labelsep {\bfseries #1}]}{\end{trivlist}}

\newcommand{\AmS}{{\protect\the\textfont2
  A\kern-.1667em\lower.5ex\hbox{M}\kern-.125emS}}

\usepackage{amsmath}
\usepackage{amsfonts}
\title{Interval non-edge-colorable bipartite graphs and multigraphs}

\author{Petros A. Petrosyan\address[MCSD]{Department of Informatics and Applied Mathematics,\\
Yerevan State University, 0025, Armenia}%
\address{Institute for Informatics and Automation Problems,\\
National Academy of Sciences, 0014, Armenia}%
\thanks{email: pet\_petros@\{ipia.sci.am, ysu.am, yahoo.com\}},
        Hrant H. Khachatrian\addressmark[MCSD]%
\thanks {email: hrant@egern.net}}

\runtitle{Interval non-edge-colorable bipartite graphs and
multigraphs}\runauthor{Petros A. Petrosyan, Hrant H. Khachatrian}

\begin{document}

\maketitle

\begin{abstract}
An edge-coloring of a graph $G$ with colors $1,\ldots,t$ is called
an interval $t$-coloring if all colors are used, and the colors of
edges incident to any vertex of $G$ are distinct and form an
interval of integers. In 1991 Erd\H{o}s constructed a bipartite
graph with $27$ vertices and maximum degree $13$ which has no
interval coloring. Erd\H{o}s's counterexample is the smallest (in a
sense of maximum degree) known bipartite graph which is not interval
colorable. On the other hand, in 1992 Hansen showed that all
bipartite graphs with maximum degree at most $3$ have an interval
coloring. In this paper we give some methods for constructing of
interval non-edge-colorable bipartite graphs. In particular, by
these methods, we construct three bipartite graphs which have no
interval coloring, contain $20,19,21$ vertices and have maximum
degree $11,12,13$, respectively. This partially answers a question
that arose in [T.R. Jensen, B. Toft, Graph coloring problems, Wiley
Interscience Series in Discrete Mathematics and Optimization, 1995,
p. 204]. We also consider similar problems for bipartite multigraphs.\\


\end{abstract}

\section{Introduction}\

In this paper we consider graphs which are finite, undirected, and
have no loops or multiple edges and multigraphs which may contain
multiple edges but no loops. Let $V(G)$ and $E(G)$ denote the sets
of vertices and edges of a multigraph $G$, respectively. For two
distinct vertices $u$ and $v$ of a multigraph $G$, let $E(uv)$
denote the set of all edges of $G$ joining $u$ with $v$, and let
$\mu (uv)$ denote the number of edges joining $u$ with $v$ (i.e.
$\mu (uv)=\vert E(uv)\vert$). The degree of a vertex $v\in V(G)$ is
denoted by $d_{G}(v)$ (or $d(v)$), the maximum degree of $G$ by
$\Delta (G)$, and the edge-chromatic number of $G$ by $\chi^{\prime
}\left( G\right)$. The terms and concepts that we do not define can
be found in \cite{b23}.

Let $G$ be a connected graph and $V(G)=\{v_{1},\ldots,v_{n}\}$,
$n\geq 2$. Let $P(v_{i},v_{j})$ be a simple path joining $v_{i}$ and
$v_{j}$, $VP(v_{i},v_{j})$ and $EP(v_{i},v_{j})$ denote the sets of
vertices and edges of this path, respectively.

A proper edge-coloring of a multigraph $G$ is a coloring of the
edges of $G$ such that no two adjacent edges receive the same color.
If $\alpha $ is a proper edge-coloring of $G$ and $v\in V(G)$, then
$S\left(v,\alpha \right)$ denotes the set of colors of edges
incident to $v$. A proper edge-coloring of a multigraph $G$ with
colors $1,\ldots ,t$ is called an interval $t$-coloring if all
colors are used, and for any vertex $v$ of $G$, the set
$S\left(v,\alpha \right)$ is an interval of integers. A multigraph
$G$ is interval colorable if it has an interval $t$-coloring for
some positive integer $t$. The set of all interval colorable
multigraphs is denoted by $\mathfrak{N}$. For a multigraph $G\in
\mathfrak{N}$, the least value of $t$ for which $G$ has an interval
$t$-coloring is denoted by $w\left(G\right)$.

The concept of interval edge-coloring of multigraphs was introduced
by Asratian and Kamalian \cite{b1,b2}. In \cite{b1,b2} they proved
that if $G$ is interval colorable, then $\chi^{\prime
}\left(G\right)=\Delta(G)$. Moreover, if $G$ is $r$-regular, then
$G$ has an interval coloring if and only if $G$ has a proper
$r$-edge-coloring. This implies that the problem \textquotedblleft
Is a given $r$-regular ($r\geq 3$) graph interval colorable or
not?\textquotedblright is $NP$-complete. Asratian and Kamalian also
proved \cite{b1,b2} that if a triangle-free graph $G$ has an
interval $t$-coloring, then $t\leq \left\vert V(G)\right\vert -1$.
In \cite{b13} Kamalian investigated interval colorings of complete
bipartite graphs and trees. In particular, he proved that the
complete bipartite graph $K_{m,n}$ has an interval $t$-coloring if
and only if $m+n-\gcd(m,n)\leq t\leq m+n-1$, where $\gcd(m,n)$ is
the greatest common divisor of $m$ and $n$. In \cite{b18} Petrosyan
investigated interval colorings of complete graphs and
$n$-dimensional cubes. In particular, he proved that if $n\leq t\leq
\frac{n\left(n+1\right)}{2}$, then the $n$-dimensional cube $Q_{n}$
has an interval $t$-coloring. In \cite{b20} Sevast'janov proved that
it is an $NP$-complete problem to decide whether a bipartite graph
has an interval coloring or not. On the other hand, computer search
in \cite{b5} showed that the following result holds.

\begin{theorem}
\label{mytheorem1} All bipartite graphs of order at most $14$ are
interval colorable.
\end{theorem}

For subcubic bipartite graphs, Hansen proved the following

\begin{theorem}
\label{mytheorem2}\cite{b10}. If $G$ is a bipartite graph with
$\Delta(G)\leq 3$, then $G\in \mathfrak{N}$ and $w(G)\leq 4$.
\end{theorem}

For bipartite graphs with maximum degree $4$, Giaro proved the
following two results:

\begin{theorem}
\label{mytheorem3}\cite{b6}. If $G$ is a bipartite graph with
$\Delta(G)=4$ and without a vertex of degree $3$, then $G\in
\mathfrak{N}$ and $w(G)=4$.
\end{theorem}

\begin{theorem}
\label{mytheorem4}\cite{b6}. The problem of deciding the existence
of interval $\Delta(G)$-coloring of a bipartite graph $G$ can be
solved in polynomial time if $\Delta(G)\leq 4$ and is $NP$-complete if
$\Delta(G)\geq 5$.
\end{theorem}

For bipartite graphs where one of the parts is small, the following
theorem was proved in \cite{b8}.

\begin{theorem}
\label{mytheorem5} If $G$ is a bipartite graph with a bipartition
$(U,V)$ and $\min \{\vert U\vert, \vert V\vert\}\leq 3$, then $G\in
\mathfrak{N}$.
\end{theorem}
Also, it is known that all regular bipartite graphs \cite{b1,b2},
doubly convex bipartite graphs \cite{b3,b14}, grids \cite{b7},
outerplanar bipartite graphs \cite{b9}, $\left(2,b\right)$-biregular
bipartite graphs \cite{b11,b15,b16} and some classes of
$\left(3,4\right)$-biregular bipartite graphs \cite{b4,b19,b24} have
interval colorings. However, there are bipartite graphs which have
no interval colorings. First example of a bipartite graph that is
not interval colorable was obtained by Mirumyan \cite{b17} in 1989,
but it was not published. The graph which was found by Mirumyan has
$19$ vertices and maximum degree $15$. First published example was
given by Sevast'janov \cite{b20} and it has $28$ vertices and
maximum degree $21$ (see Fig. \ref{fig1}). Other examples were
obtained by Erd\H{o}s ($27$ vertices and maximum degree $13$), by
Hertz and de Werra ($21$ vertices and maximum degree $14$), and by
Malafiejski ($19$ vertices and maximum degree $15$). In \cite{b12},
Jensen and Toft posed the following question:

\begin{figure}[h]
\begin{center}
\includegraphics[width=25pc]{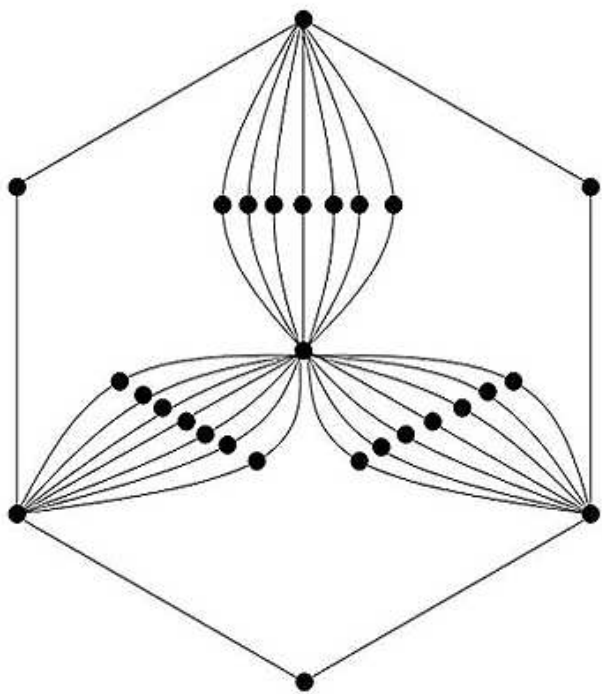}\\
\caption{The Sevast'janov graph.}\label{fig1}
\end{center}
\end{figure}

\begin{problem}
Is there a bipartite graph G with $4\leq \Delta(G)\leq 12$ and
$G\notin \mathfrak{N}$?
\end{problem}

In the present paper we describe some methods for constructing of
interval non-edge-colorable bipartite graphs. In particular, by
these methods, we construct two bipartite graphs $G$ and $H$ with
$\Delta(G)=11$, $\Delta(H)=12$ which have no interval coloring. This
partially answers a question of Jensen and Toft. In this paper we
also consider similar problems for bipartite multigraphs.

\bigskip

\section{Interval non-edge-colorable bipartite graphs}\

\subsection{Counterexamples by fat triangles}\

In 1949 Shannon \cite{b21} proved that $\chi^{\prime }(G)\leq
\left\lfloor \frac{3}{2}\Delta(G)\right\rfloor$ for any multigraph
$G$. Also, he showed that this upper bound is sharp for special
multigraphs which are called fat triangles. The fat triangle is a
multigraph with three vertices $x,y,z$ and $r$ edges between each
pair of vertices, that is, $\mu (xy)=\mu (yz)=\mu (xz)=r$. Later,
Vizing \cite{b22} proved that if a multigraph $G$ has $\chi^{\prime
}(G)=\left\lfloor \frac{3}{2}\Delta(G)\right\rfloor$ and
$\Delta(G)\geq 4$, then $G$ has a fat triangle as a subgraph. In
this paragraph we use fat triangles for constructing of interval
non-edge-colorable bipartite graphs. First note that the graph
obtained by subdividing every edge of a fat triangle is bipartite.
Moreover, a new graph obtained from the subdivided graph by
connecting every inserted vertex to a new vertex is also bipartite.
Now let us define the graph $\Delta_{r,s,t}$ ($1\leq r\leq s\leq t$)
as follows:
\begin{center}
$V(\Delta_{r,s,t})=\{v,x,y,z\}\cup\{a_{1},\ldots,a_{r},b_{1},\ldots,b_{s},c_{1},\ldots,c_{t}\}$,\\
\end{center}
\begin{center}
$E(\Delta_{r,s,t})=\{va_{i},xa_{i},ya_{i}:1\leq i\leq r\}\cup
\{vb_{j},xb_{j},zb_{j}:1\leq j\leq s\}\cup
\{vc_{k},yc_{k},zc_{k}:1\leq k\leq t\}$.
\end{center}

Clearly, $\Delta_{r,s,t}$ is a connected bipartite graph with $\vert
V(\Delta_{r,s,t})\vert=r+s+t+4$, $d(x)=r+s$, $d(y)=r+t$, $d(z)=s+t$,
and $\Delta(\Delta_{r,s,t})=r+s+t$. Note that our $\Delta_{r,s,t}$
graphs generalize Malafiejski's rosettes $M_{k}$ given in \cite{b8},
since $M_{k}=\Delta_{k,k,k}$ for any $k\in \mathbf{N}$.

\begin{theorem}
\label{mytheorem2.1.1} If $r\geq 5$, then $\Delta_{r,s,t}\notin
\mathfrak{N}$.
\end{theorem}
\begin{proof}
Suppose, to the contrary, that the graph $\Delta_{r,s,t}$ has an
interval $q$-coloring $\alpha$ for some $q\geq r+s+t$.

Consider the vertex $v$. Let $u$ and $w$ be two vertices adjacent to
$v$ such that $\alpha(vu)={\min}~S(v,\alpha)=p$ and
$\alpha(vw)={\max}~S(v,\alpha)=p+r+s+t-1$. By the construction of
$\Delta_{r,s,t}$, there is a path $P(u,w)$ in $\Delta_{r,s,t}-v$ of
length two joining $u$ with $w$, where
\begin{center}
$P(u,w)=(u,uv^{\prime},v^{\prime},v^{\prime}w,w)$.
\end{center}

Since $d(u)=3$ and $d(v^{\prime})\leq s+t$, we have

\begin{center}
$\alpha(uv^{\prime})\leq p+d(u)-1=p+2$ and thus
\end{center}
\begin{center}
$\alpha(v^{\prime}w)\leq p+2+ d(v^{\prime})-1=p+1+s+t$.
\end{center}

On the other hand, since $d(w)=3$, we have

\begin{center}
$p+r+s+t-1=\alpha(vw)=\max S(v,\alpha)\leq p+1+s+t+d(w)-1=p+s+t+3$.
\end{center}
Hence, $r\leq 4$, which is a contradiction. ~$\square$
\end{proof}

\begin{figure}[h]
\begin{center}
\includegraphics[width=30pc]{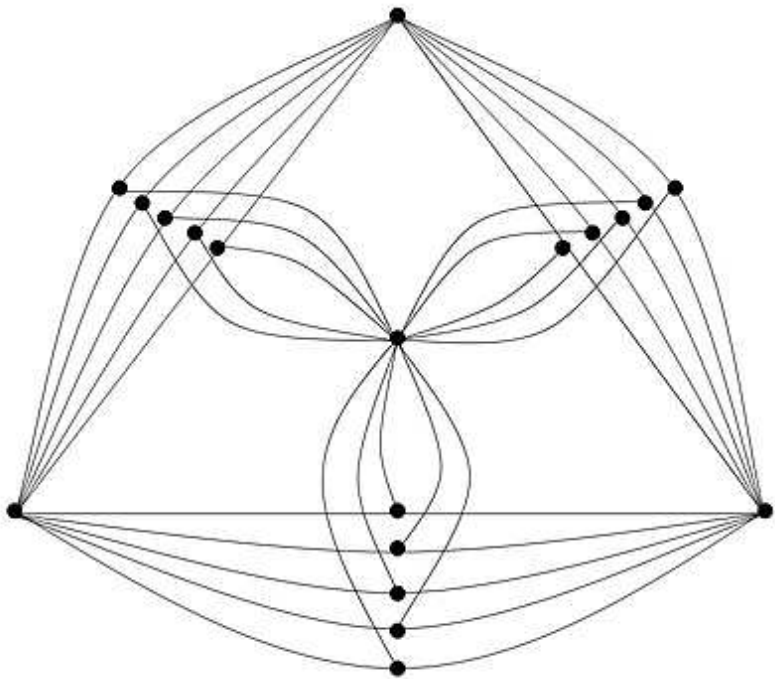}\\
\caption{The graph $\Delta_{5,5,5}$.}\label{fig2}
\end{center}
\end{figure}

Theorem \ref{mytheorem2.1.1} implies that the graph $\Delta_{5,5,5}$
with $\vert V(\Delta_{5,5,5})\vert=19$ and
$\Delta(\Delta_{5,5,5})=15$ shown in Fig. \ref{fig2} has no interval
coloring. In fact, this is the example of an interval
non-edge-colorable bipartite graph that first was constructed by
Mirumyan. This example first appeared in \cite{b3} and \cite{b8},
and currently is known as Malafiejski's rosette $M_{5}$.

\begin{corollary}
\label{mycorollary2.1.1} For any positive integer $\Delta \geq 15$,
there is a connected bipartite graph $G$ with $G\notin \mathfrak{N}$
and $\Delta(G)=\Delta$.
\end{corollary}

\bigskip

\subsection{Counterexamples by finite projective planes}\

In this paragraph we use finite projective planes for constructing
of interval non-edge-colorable bipartite graphs. A finite projective
plane $\pi(n)$ of order $n$ ($n\geq 2$) has $n^{2}+n+1$ points and
$n^{2}+n+1$ lines, and satisfies the following properties:

\begin{description}
\item[P1] any two points determine a line;

\item[P2] any two lines determine a point;

\item[P3] every point is incident to $n+1$ lines;

\item[P4] every line is incident to $n+1$ points.
\end{description}

Let $\{1,\ldots,n^{2}+n+1\}$ be the set of points and $L$ be the set
of lines of $\pi(n)$. Define the graph
$Erd(r_{1},\ldots,r_{n^{2}+n+1})$ ($r_{1}\geq \ldots\geq
r_{n^{2}+n+1}\geq 1$) as follows:

\begin{center}
$V(Erd(r_{1},\ldots,r_{n^{2}+n+1}))=\{u\}\cup\{1,\ldots,n^{2}+n+1\}$\\
$\cup \left\{v^{(l_{i})}_{1},\ldots,v^{(l_{i})}_{r_{i}}:l_{i}\in L,1\leq i\leq n^{2}+n+1\right\}$,\\
\end{center}
\begin{center}
$E(Erd(r_{1},\ldots,r_{n^{2}+n+1}))=\left\{uv^{(l_{i})}_{1},\ldots,uv^{(l_{i})}_{r_{i}}:l_{i}\in
L, 1\leq i\leq
n^{2}+n+1\right\}\cup$\\
$\bigcup_{i=1}^{n^{2}+n+1}\left\{v^{(l_{i})}_{1}k,\ldots,v^{(l_{i})}_{r_{i}}k:l_{i}\in
L,k\in l_{i},1\leq k\leq n^{2}+n+1\right\}$.
\end{center}

Clearly, $Erd(r_{1},\ldots,r_{n^{2}+n+1})$ is a connected bipartite
graph with $\Delta
(Erd(r_{1},\ldots,r_{n^{2}+n+1}))=\underset{i=1}{\overset{n^{2}+n+1}{\sum
}}r_{i}$ and $\vert
V(Erd(r_{1},\ldots,r_{n^{2}+n+1}))\vert=\underset{i=1}{\overset{n^{2}+n+1}{\sum
}}r_{i}+n^{2}+n+2$. Note that the graph
$Erd(1,1,1,1,1,1,1,1,1,1,1,1,1)$ was described by Erd\H{o}s in 1991
\cite{b12}. This graph has $27$ vertices and maximum degree $13$.

\begin{theorem}
\label{mytheorem2.2.1} If $\underset{i=n+2}{\overset{n^{2}+n+1}{\sum
}}r_{i}> 2(n+1)$, then $Erd(r_{1},\ldots,r_{n^{2}+n+1})\notin
\mathfrak{N}$.
\end{theorem}
\begin{proof}
Suppose, to the contrary, that the graph
$Erd(r_{1},\ldots,r_{n^{2}+n+1})$ has an interval $t$-coloring
$\alpha$ for some $t\geq \underset{i=1}{\overset{n^{2}+n+1}{\sum
}}r_{i}$.

Consider the vertex $u$. Let $v_{p}^{(l_{i_{0}})}$ and
$v_{q}^{(l_{j_{0}})}$ be two vertices adjacent to $u$ such that
$\alpha\left(uv_{p}^{(l_{i_{0}})}\right)={\min}~S(u,\alpha)=s$ and
$\alpha\left(uv_{q}^{(l_{j_{0}})}\right)={\max}~S(u,\alpha)=s+\underset{i=1}{\overset{n^{2}+n+1}{\sum
}}r_{i}-1$.

If $l_{i_{0}}=l_{j_{0}}$, then, by the construction of
$Erd(r_{1},\ldots,r_{n^{2}+n+1})$, there exists $k_{0}$ such that
$k_{0}v_{p}^{(l_{i_{0}})}$, $k_{0}v_{q}^{(l_{j_{0}})}\in
E(Erd(r_{1},\ldots,r_{n^{2}+n+1}))$. If $l_{i_{0}}\neq l_{j_{0}}$,
then, by the construction of $Erd(r_{1},\ldots,r_{n^{2}+n+1})$ and
the property P2, there exists $k_{0}$ such that
$k_{0}v_{p}^{(l_{i_{0}})}$, $k_{0}v_{q}^{(l_{j_{0}})}\in
E(Erd(r_{1},\ldots,r_{n^{2}+n+1}))$.

By the construction of $Erd(r_{1},\ldots,r_{n^{2}+n+1})$ and
properties P3 and P4, we have
$d\left(v_{p}^{(l_{i_{0}})}\right)=d\left(v_{q}^{(l_{j_{0}})}\right)=n+2$
and
\begin{center}
$\alpha\left(k_{0}v_{p}^{(l_{i_{0}})}\right)\leq
s+d\left(v_{p}^{(l_{i_{0}})}\right)-1=s+n+1$ and thus
\end{center}
\begin{center}
$\alpha\left(k_{0}v_{q}^{(l_{j_{0}})}\right)\leq
s+n+1+d(k_{0})-1\leq s+n+\underset{i=1}{\overset{n+1}{\sum }}r_{i}$.
\end{center}

This implies that
\begin{center}
$s+\underset{i=1}{\overset{n^{2}+n+1}{\sum
}}r_{i}-1=\alpha\left(uv_{q}^{(l_{j_{0}})}\right)={\max}~S(u,\alpha)\leq
s+n+\underset{i=1}{\overset{n+1}{\sum
}}r_{i}+d\left(v_{q}^{(l_{j_{0}})}\right)-1=s+2n+1+\underset{i=1}{\overset{n+1}{\sum
}}r_{i}$.
\end{center}

Hence, $\underset{i=n+2}{\overset{n^{2}+n+1}{\sum }}r_{i}\leq
2(n+1)$, which is a contradiction.
 ~$\square$
\end{proof}

\begin{corollary}
\label{mycorollary2.2.1} For any positive integer $\Delta \geq 13$,
there is a connected bipartite graph $G$ with $G\notin \mathfrak{N}$
and $\Delta(G)=\Delta$.
\end{corollary}

\begin{figure}[h]
\begin{center}
\includegraphics[width=30pc]{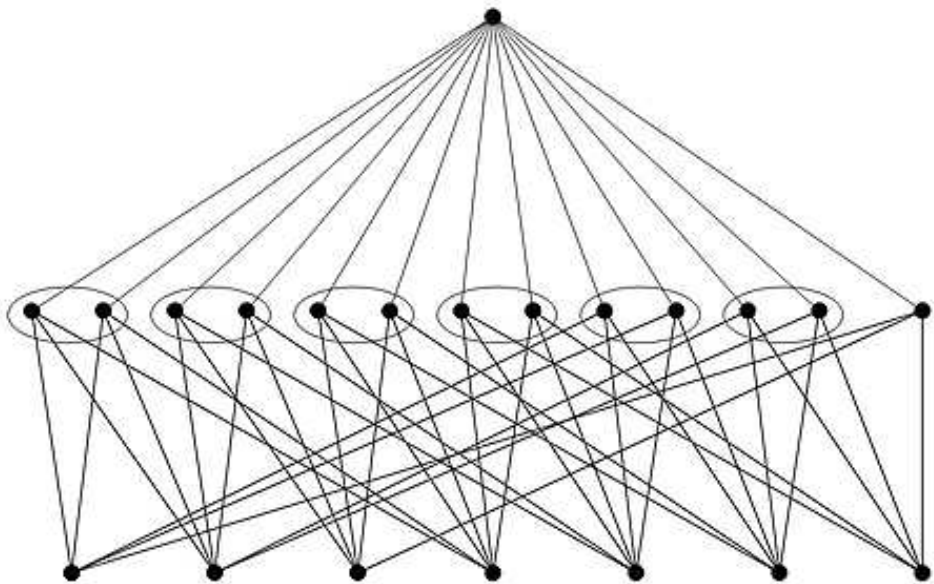}\\
\caption{The graph $Erd(2,2,2,2,2,2,1)$.}\label{fig3}
\end{center}
\end{figure}

Theorem \ref{mytheorem2.2.1} implies that the graph
$Erd(2,2,2,2,2,2,1)$ with $\vert V(Erd(2,2,2,2,2,2,1))\vert = 21$
and $\Delta(Erd(2,2,2,2,2,2,1))= 13$ shown in Fig. \ref{fig3} has no
interval coloring. Also, Theorem \ref{mytheorem2.2.1} implies that
the graph $Erd(2,2,2,2,2,2,2)$ with $\vert
V(Erd(2,2,2,2,2,2,2))\vert =22$ and $\Delta(Erd(2,2,2,2,2,2,1)) =14$
has no interval coloring. In the next section we show that there is
a connected bipartite graph $G$ with $\vert V(G)\vert =21$ and
$\Delta(G)=14$ which is not interval colorable.

\bigskip

\subsection{Counterexamples by trees}\

Let $T$ be a tree and $V(T)=\{v_{1},\ldots,v_{n}\}$, $n\geq 2$. For
a simple path $P(v_{i},v_{j})$, define $L(v_{i},v_{j})$ as follows:
\begin{center}
$L(v_{i},v_{j})=\vert EP(v_{i},v_{j})\vert +\vert\left\{uw:uw\in
E(T), u\in VP(v_{i},v_{j}), w\notin VP(v_{i},v_{j})\right\}\vert$.
\end{center}

Define:
\begin{center}
$M(T)={\max }_{1\leq i\leq n, 1\leq j\leq n}L(v_{i},v_{j})$.
\end{center}

In \cite{b14}, Kamalian proved the following result.

\begin{theorem}
\label{mytheorem2.3.1} If $T$ is a tree, then $T$ has an interval
$t$-coloring if and only if $\Delta(T)\leq t\leq M(T)$.
\end{theorem}

Now let $T$ be a tree in which the distance between any two pendant
vertices is even and $F(T)=\{v:v\in V(T)\wedge d_{T}(v)=1\}$.

Let us define the graph $\widetilde{T}$ as follows:
\begin{center}
$V(\widetilde{T})=V(T)\cup \{u\}$, $u\notin V(T)$,
$E(\widetilde{T})=E(T)\cup \{uv:v\in F(T)\}$.
\end{center}

Clearly, $\widetilde{T}$ is a connected bipartite graph with
$\Delta(\widetilde{T})=\vert F(T)\vert$.

\begin{theorem}
\label{mytheorem2.3.2} If $T$ is a tree in which the distance
between any two pendant vertices is even and $\vert F(T)\vert >
M(T)+2$, then $\widetilde{T}\notin \mathfrak{N}$.
\end{theorem}
\begin{proof}
Suppose, to the contrary, that $\widetilde{T}$ has an interval
$t$-coloring $\alpha$ for some $t\geq \vert F(T)\vert$.

Consider the vertex $u$. Let $v$ and $v^{\prime}$ be two vertices
adjacent to $u$ such that $\alpha(uv)={\min}~S(u,\alpha)=s$ and
$\alpha(uv^{\prime})={\max}~S(u,\alpha)=s+\vert F(T)\vert-1$. Since
$\widetilde{T}-u$ is a tree, there is a unique path
$P(v,v^{\prime})$ in $\widetilde{T}-u$ joining $v$ with
$v^{\prime}$, where
\begin{center}
$P(v,v^{\prime})=(v_{1},e_{1},v_{2},\ldots,v_{i},e_{i},v_{i+1},\ldots,v_{k},e_{k},v_{k+1})$,
$v_{1}=v$, $v_{k+1}=v^{\prime}$.
\end{center}

Note that
\begin{center}
$\alpha(v_{i}v_{i+1})\leq s+1+\underset{j=1}{\overset{i}{\sum
}}(d_{T}(v_{j})-1)$ for $1\leq i\leq k$.
\end{center}

From this, we have
\begin{center}
$\alpha(v_{k}v_{k+1})=\alpha(v_{k}v^{\prime})\leq
s+1+\underset{j=1}{\overset{k}{\sum
}}(d_{T}(v_{j})-1)=s+L(v,v^{\prime})\leq s+M(T)$.
\end{center}

Hence
\begin{center}
$s+\vert F(T)\vert -1={\max}~S(u,\alpha)=\alpha(uv^{\prime})\leq
s+1+M(T)$ and thus $\vert F(T)\vert\leq M(T)+2$,
\end{center}
which is a contradiction.
 ~$\square$
\end{proof}

Now let us consider the tree $T$ shown in Fig. \ref{fig4}.

\begin{figure}[h]
\begin{center}
\includegraphics[width=20pc]{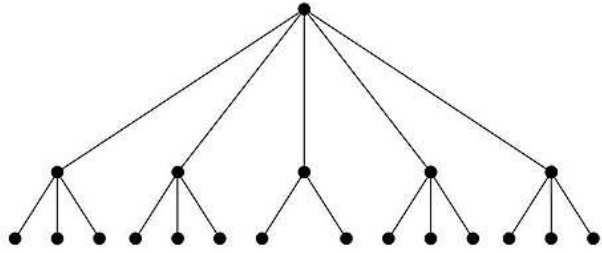}\\
\caption{The tree $T$.}\label{fig4}
\end{center}
\end{figure}

Since $M(T)=11$ and $\vert F(T)\vert=14$, the graph $\widetilde{T}$
with $\vert V(\widetilde{T})\vert=21$ and $\Delta(\widetilde{T})=14$
has no interval coloring. Our constructions by trees generalize
Hertz's graphs $H_{p,q}$ given in \cite{b8}. Moreover, the
aforementioned example obtained by the method described above is
smaller than the smallest Hertz's graph $H_{7,2}$.

\bigskip

\subsection{Counterexamples by subdivisions}\

In this section we also need a definition of the interval of
positive integers. For positive integers $a$ and $b$, we denote by
$\left[a,b\right]$, the set of all positive integers $c$ with $a\leq
c\leq b$.

Let $G$ be a graph and $V(G)=\{v_{1},\ldots,v_{n}\}$. Define graphs
$S(G)$ and $\widehat{G}$ as follows:
\begin{center}
$V(S(G))=\{v_{1},\ldots,v_{n}\}\cup \{w_{ij}:v_{i}v_{j}\in E(G)\}$,
\end{center}
\begin{center}
$E(S(G))=\{v_{i}w_{ij},v_{j}w_{ij}:v_{i}v_{j}\in E(G)\}$,
\end{center}
\begin{center}
$V(\widehat{G})=V(S(G))\cup \{u\}$, $u\notin V(S(G))$,
$E(\widehat{G})=E(S(G))\cup \{uw_{ij}:v_{i}v_{j}\in E(G)\}$.
\end{center}

In other words, $S(G)$ is the graph obtained by subdividing every
edge of $G$, and $\widehat{G}$ is the graph obtained from $S(G)$ by
connecting every inserted vertex to a new vertex $u$. Clearly,
$S(G)$ and $\widehat{G}$ are bipartite graphs.

\begin{proposition}
\label{myproposition2.4.1} If $G$ is a bipartite graph and $G\in
\mathfrak{N}$, then $S(G)\in \mathfrak{N}$.
\end{proposition}
\begin{proof}
Let $G$ be a bipartite graph with a bipartition $(U,V)$, where
$U=\{u_{1},\ldots,u_{r}\}$, $V=\{v_{1},\ldots,v_{s}\}$. Also, let
$\alpha$ be an interval $t$-coloring of the graph $G$.

Define an edge-coloring $\beta$ of the graph $S(G)$ as follows:
\begin{center}
$\beta(u_{i}w_{ij})=\alpha(u_{i}v_{j})$ and
$\beta(v_{j}w_{ij})=\alpha(u_{i}v_{j})+1$ for every $u_{i}v_{j}\in
E(G)$.
\end{center}

It is easy to see that $\beta$ is an interval $(t+1)$-coloring of
the graph $S(G)$.
 ~$\square$
\end{proof}

In \cite{b11,b15,b16}, it was proved that if $G$ is a regular graph,
then $S(G)\in \mathfrak{N}$. It would be interesting to generalize
the last two statements to general graphs. In other words, we would
like to suggest the following

\begin{conjecture}
If $G$ is a simple graph and $G\in \mathfrak{N}$, then $S(G)\in
\mathfrak{N}$.
\end{conjecture}

\begin{theorem}
\label{mytheorem2.4.1} If $G$ is a connected graph and
\begin{center}
$\vert E(G)\vert
> 1+ {\max\limits_{P\in \mathbf{P}}}{\sum\limits_{v\in
V(P)}}\ \left(d_{\widehat{G}}(v)-1\right)$,
\end{center}
where $\mathbf{P}$ is a set of all shortest paths in $S(G)$
connecting vertices $w_{ij}$, then $\widehat{G}\notin \mathfrak{N}$.
\end{theorem}
\begin{proof}
Suppose, to the contrary, that $\widehat{G}$ has an interval
$t$-coloring $\alpha$ for some $t\geq \vert E(G)\vert$.

Consider the vertex $u$. Let $w$ and $w^{\prime}$ be two vertices
adjacent to $u$ such that $\alpha(uw)={\min}~S(u,\alpha)=s$ and
$\alpha(uw^{\prime})={\max}~S(u,\alpha)=s+\vert E(G)\vert-1$. Since
$\widehat{G}-u$ is isomorphic to $S(G)$ and connected, there is a
shortest path $P(w,w^{\prime})$ in $\widehat{G}-u$ joining $w$ with
$w^{\prime}$, where
\begin{center}
$P(w,w^{\prime})=(v_{1},e_{1},v_{2},\ldots,v_{i},e_{i},v_{i+1},\ldots,v_{k},e_{k},v_{k+1})$,
$v_{1}=w$, $v_{k+1}=w^{\prime}$.
\end{center}

Note that
\begin{center}
$\alpha(v_{i}v_{i+1})\leq s+\underset{j=1}{\overset{i}{\sum
}}(d_{\widehat{G}}(v_{j})-1)$ for $1\leq i\leq k$
\end{center}
and
\begin{center}
$\alpha(v_{k+1}u)=\alpha(w^{\prime}u)\leq
s+\underset{j=1}{\overset{k+1}{\sum }}(d_{\widehat{G}}(v_{j})-1)$.
\end{center}

Hence
\begin{center}
$s+\vert E(G)\vert -1={\max}~S(u,\alpha)=\alpha(uw^{\prime})\leq
s+\underset{j=1}{\overset{k+1}{\sum }}(d_{\widehat{G}}(v_{j})-1)\leq
s+{\max\limits_{P\in \mathbf{P}}}{\sum\limits_{v\in V(P)}}\
\left(d_{\widehat{G}}(v)-1\right)$
\end{center}
and thus
\begin{center}
$\vert E(G)\vert \leq 1+{\max\limits_{P\in
\mathbf{P}}}{\sum\limits_{v\in V(P)}}\
\left(d_{\widehat{G}}(v)-1\right)$,
\end{center}
which is a contradiction.
 ~$\square$
\end{proof}

\begin{corollary}
\label{mycorollary2.4.1} If $n\geq 7$, then $\widehat{K}_{n}\notin
\mathfrak{N}$.
\end{corollary}

\begin{corollary}
\label{mycorollary2.4.2} If $mn-m-n>5$, then
$\widehat{K}_{m,n}\notin \mathfrak{N}$.
\end{corollary}

\begin{figure}[h]
\begin{center}
\includegraphics[width=25pc]{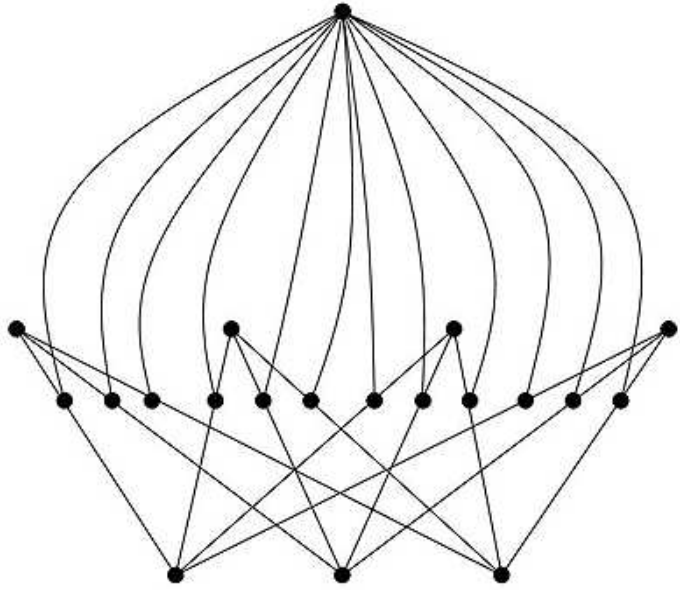}\\
\caption{The graph $\widehat{K}_{3,4}$.}\label{fig5}
\end{center}
\end{figure}

Now we show that the graph $\widehat{K}_{3,4}$ shown in Fig.
\ref{fig5} has no interval coloring.

\begin{theorem}
\label{mytheorem2.4.2} $\widehat{K}_{3,4}\notin \mathfrak{N}$.
\end{theorem}
\begin{proof}
Let
$V(\widehat{K}_{3,4})=\{u,v_{1},v_{2},v_{3},v_{4},v_{5},v_{6},v_{7}\}\cup
\{w_{ij}:1\leq i\leq 3, 4\leq j\leq 7\}$ and
$E(\widehat{K}_{3,4})=\{v_{i}w_{ij},v_{j}w_{ij},uw_{ij}:1\leq i\leq
3, 4\leq j\leq 7\}$.

Suppose that $\widehat{K}_{3,4}$ has an interval $t$-coloring
$\alpha$ for some $t\geq 12$.

Consider the vertex $u$. Let $w_{i_{0}j_{0}}$ and $w_{i_{1}j_{1}}$
be two vertices adjacent to $u$ such that
$\alpha(uw_{i_{0}j_{0}})={\min}~S(u,\alpha)=s$ and
$\alpha(uw_{i_{1}j_{1}})={\max}~S(u,\alpha)=s+11$. We consider two
cases.

Case 1: $i_{0}=i_{1}$ or $j_{0}=j_{1}$.

If $i_{0}=i_{1}$, then
$v_{i_{0}}w_{i_{0}j_{0}},v_{i_{0}}w_{i_{0}j_{1}}\in
E(\widehat{K}_{3,4})$. This implies that

\begin{center}
$\alpha(v_{i_{0}}w_{i_{0}j_{0}})\leq s+2$ and
$\alpha(v_{i_{0}}w_{i_{0}j_{1}})\leq s+5$.
\end{center}

Hence,
\begin{center}
$s+11={\max}~S(u,\alpha)=\alpha(uw_{i_{0}j_{1}})\leq s+7$,
\end{center}
which is impossible.

If $j_{0}=j_{1}$, then
$v_{j_{0}}w_{i_{0}j_{0}},v_{j_{0}}w_{i_{1}j_{0}}\in
E(\widehat{K}_{3,4})$. This implies that

\begin{center}
$\alpha(v_{j_{0}}w_{i_{0}j_{0}})\leq s+2$ and
$\alpha(v_{j_{0}}w_{i_{1}j_{0}})\leq s+4$.
\end{center}

Hence,
\begin{center}
$s+11={\max}~S(u,\alpha)=\alpha(uw_{i_{1}j_{0}})\leq s+6$,
\end{center}
which is impossible.

Case 2: $i_{0}\neq i_{1}$ and $j_{0}\neq j_{1}$.

In this case the edges $v_{i_{0}}v_{j_{0}}$ and  $v_{i_{1}}v_{j_{1}}$
are independent in $K_{3,4}$. Clearly, any two independent edges in
$K_{3,4}$ lie on the cycle of a length four. Hence, there is a cycle
$C=w_{i_{0}j_{0}}v_{j_{0}}w_{i_{1}j_{0}}v_{i_{1}}w_{i_{1}j_{1}}v_{j_{1}}w_{i_{0}j_{1}}v_{i_{0}}w_{i_{0}j_{0}}$
in $\widehat{K}_{3,4}$, which is consists of paths $P$ and $Q$,
where
\begin{center}
$P=\left(w_{i_{0}j_{0}},v_{j_{0}}w_{i_{0}j_{0}},v_{j_{0}},v_{j_{0}}w_{i_{1}j_{0}},w_{i_{1}j_{0}},v_{i_{1}}w_{i_{1}j_{0}},v_{i_{1}},v_{i_{1}}w_{i_{1}j_{1}},w_{i_{1}j_{1}}\right)$
\end{center}
and
\begin{center}
$Q=\left(w_{i_{0}j_{0}},v_{i_{0}}w_{i_{0}j_{0}},v_{i_{0}},v_{i_{0}}w_{i_{0}j_{1}},w_{i_{0}j_{1}},v_{j_{1}}w_{i_{0}j_{1}},v_{j_{1}},v_{j_{1}}w_{i_{1}j_{1}},w_{i_{1}j_{1}}\right)$.
\end{center}

If $\alpha(v_{j_{0}}w_{i_{0}j_{0}})=s+1$, then, by considering the
path $P$, we have $\alpha(v_{i_{1}}w_{i_{1}j_{1}})\leq s+8$ and
${\max}~S(w_{i_{1}j_{1}},\alpha)\leq s+10$, a contradiction.

If $\alpha(v_{i_{0}}w_{i_{0}j_{0}})=s+1$, then, by considering the
path $Q$, we have $\alpha(v_{j_{1}}w_{i_{1}j_{1}})\leq s+8$ and
${\max}~S(w_{i_{1}j_{1}},\alpha)\leq s+10$, a contradiction.

Hence,
$\alpha(v_{j_{0}}w_{i_{0}j_{0}})=\alpha(v_{i_{0}}w_{i_{0}j_{0}})=s+2$,
which is a contradiction.
 ~$\square$
\end{proof}

\begin{figure}[h]
\begin{center}
\includegraphics[width=25pc]{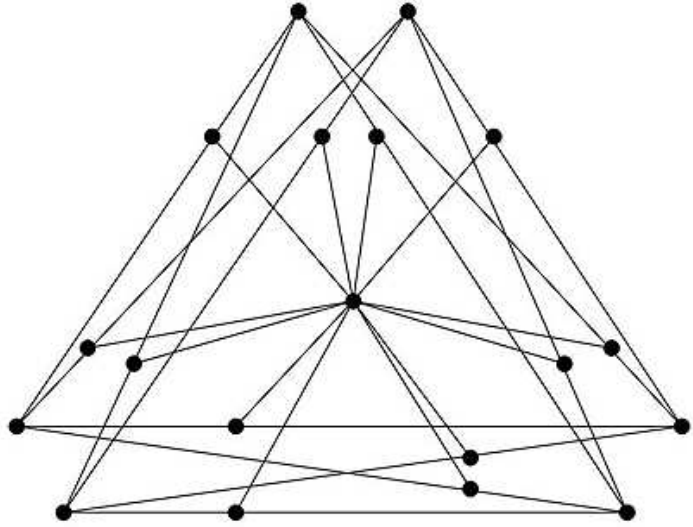}\\
\caption{The graph $\widehat{K}_{2,2,2}$.}\label{fig6}
\end{center}
\end{figure}

Note that the graph $\widehat{K}_{3,4}$ has $20$ vertices and
maximum degree $12$. Now we show that there is a connected bipartite
graph $G$ with $\vert V(G)\vert =19$ and $\Delta(G)=12$ which is not
interval colorable. Let $K_{2,2,2}$ be a complete $3$-partite graph
with two vertices in each part. Then the graph $\widehat{K}_{2,2,2}$
shown in Fig. \ref{fig6} is not interval colorable.

\begin{theorem}
\label{mytheorem2.4.3} $\widehat{K}_{2,2,2}\notin \mathfrak{N}$.
\end{theorem}
\begin{proof}
Let
$V(\widehat{K}_{2,2,2})=\{u,v_{1},v_{2},v_{3},v_{4},v_{5},v_{6}\}\cup
\{w_{ij}:1\leq i<j\leq 6,(i,j)\notin \{(1,2),(3,4),\\
(5,6)\}\}$ and
$E(\widehat{K}_{2,2,2})=\{v_{i}w_{ij},v_{j}w_{ij},uw_{ij}:1\leq
i<j\leq 6,(i,j)\notin \{(1,2),(3,4),(5,6)\}\}$.

Suppose that $\widehat{K}_{2,2,2}$ has an interval $t$-coloring
$\alpha$ for some $t\geq 12$.

Consider the vertex $u$. Let $w_{i_{0}j_{0}}$ and $w_{i_{1}j_{1}}$
be two vertices adjacent to $u$ such that
$\alpha(uw_{i_{0}j_{0}})={\min}~S(u,\alpha)=s$ and
$\alpha(uw_{i_{1}j_{1}})={\max}~S(u,\alpha)=s+11$. By the symmetry
of the graph $\widehat{K}_{2,2,2}$, we may assume that
$(i_{0},j_{0})=(1,6)$. We consider two cases.

Case 1: $v_{1}v_{6}$ and $v_{i_{1}}v_{j_{1}}$ are adjacent in
$K_{2,2,2}$.

By the symmetry of the graph $\widehat{K}_{2,2,2}$, it suffices to
consider $(i_{1},j_{1})=(1,4)$ and $(i_{1},j_{1})=(1,5)$.

If $\alpha(uw_{14})=s+11$ or $\alpha(uw_{15})=s+11$, then
$\alpha(v_{1}w_{16})\leq s+2$ and $\alpha(v_{1}w_{1j_{1}})\leq s+5$.
Hence, $\alpha(uw_{1j_{1}})\leq s+7$, which is impossible.

Case 2: $v_{1}v_{6}$ and $v_{i_{1}}v_{j_{1}}$ are independent in
$K_{2,2,2}$.

By the symmetry of the graph $\widehat{K}_{2,2,2}$, it suffices to
consider $(i_{1},j_{1})=(4,5)$ and $(i_{1},j_{1})=(2,5)$.

If $\alpha(uw_{45})=s+11$, then either $\alpha(v_{1}w_{16})=s+1$ or
$\alpha(v_{1}w_{16})=s+2$, and in both cases the colors of all edges
along the cycle
$C=w_{16}v_{1}w_{15}v_{5}w_{45}v_{4}w_{46}v_{6}w_{16}$ are known.
This implies that $\alpha(v_{1}w_{14})\leq s+4$, but
$\alpha(v_{4}w_{14})\geq s+7$, which is a contradiction.

If $\alpha(uw_{25})=s+11$, then, by the symmetry of the graph
$\widehat{K}_{2,2,2}$, we may assume that $\alpha(v_{1}w_{16})=s+1$.
Clearly, in this case the colors of all edges along the cycle
$C=w_{16}v_{1}w_{15}v_{5}w_{25}v_{2}w_{26}v_{6}w_{16}$ are known.
This implies that $\alpha(uw_{26})=s+6$. By the symmetry of the
graph $\widehat{K}_{2,2,2}$, we may assume that
$\alpha(v_{1}w_{13})=s+2$ and $\alpha(v_{1}w_{14})=s+3$. It is easy
to see that $\alpha(uw_{13})=s+1$. Hence, $\alpha(v_{3}w_{13})=s+3$
and taking into account that $\alpha(v_{2}w_{25})=s+10$, we have
$\alpha(v_{3}w_{23})=s+6$. This implies that
$\alpha(v_{3}w_{35})\leq s+5$. On the other hand, since
$\alpha(v_{5}w_{25})=s+9$, we have $\alpha(v_{5}w_{35})\geq s+7$ and
thus $\alpha(uw_{35})=s+6=\alpha(uw_{26})$, which is a
contradiction. ~$\square$
\end{proof}

Also, we investigate bipartite graphs which are close to the graph
$\widehat{K}_{3,4}$. In particular, we observe that the graph
obtained from $\widehat{K}_{3,4}$ by deleting any edge incident to
the vertex of maximum degree is not interval colorable, too. Let
$\widehat{K}_{3,4}^{\prime}$ be a graph obtained from
$\widehat{K}_{3,4}$ by deleting any edge incident to the vertex $u$.
Clearly, $\widehat{K}_{3,4}^{\prime}$ has $20$ vertices and maximum
degree $11$.

\begin{theorem}
\label{mytheorem2.4.4} $\widehat{K}_{3,4}^{\prime}\notin
\mathfrak{N}$.
\end{theorem}
\begin{proof} Let $V(\widehat{K}_{3,4}^{\prime})= V(\widehat{K}_{3,4})$
and $E(\widehat{K}_{3,4}^{\prime})=E(\widehat{K}_{3,4})\setminus
{uw_{l_{0}m_{0}}}$.

Suppose that $\widehat{K}_{3,4}^{\prime}$ has an interval
$t$-coloring $\alpha$ for some $t\geq 11$.

Consider the vertex $u$. Let $w_{i_{0}j_{0}}$ and $w_{i_{1}j_{1}}$
be two vertices adjacent to $u$ such that
$\alpha(uw_{i_{0}j_{0}})={\min}~S(u,\alpha)=s$ and
$\alpha(uw_{i_{1}j_{1}})={\max}~S(u,\alpha)=s+10$.

Let $S(w_{l_{0}m_{0}},\alpha)=\{c,c+1\}$. Now we add the edge
$uw_{l_{0}m_{0}}$ to the graph $\widehat{K}_{3,4}^{\prime}$ and we
color it with color $c+2$. Clearly, we obtained an edge-coloring of
the graph $\widehat{K}_{3,4}$ with colors $1,\ldots,t^{\prime}
(t^{\prime}\geq t)$. Let $\beta$ be this edge-coloring. Note that
for each vertex $v\in V(\widehat{K}_{3,4})\setminus \{u\}$,
$S(v,\beta)$ is an interval of integers, and
$S(u,\beta)=[s,s+10]\cup \{c+2\}$ is a multiset in general.

Similarly as in the proof of the case 1 of Theorem
\ref{mytheorem2.4.2} it can be shown that $i_{0}\neq i_{1}$ and
$j_{0}\neq j_{1}$. Clearly, the edges $v_{i_{0}}v_{j_{0}}$ and
$v_{i_{1}}v_{j_{1}}$ are independent in $K_{3,4}$. By the symmetry
of the graph $\widehat{K}_{3,4}$, we may assume that
$(i_{0},j_{0})=(1,4)$ and $(i_{1},j_{1})=(3,7)$.

Consider the edge $v_{1}w_{14}$. Clearly, either
$\beta(v_{1}w_{14})=s+1$ or $\beta(v_{1}w_{14})=s+2$. If
$\beta(v_{1}w_{14})=s+1$, then the colors of all edges along the
cycle $C=w_{14}v_{1}w_{17}v_{7}w_{37}v_{3}w_{34}v_{4}w_{14}$ are
known. This implies that $\beta(uw_{17})=\beta(uw_{34})=s+5$. Hence,
the added color $c+2$ is $s+5$, but this is a contradiction, since
$S(w_{17},\beta)=S(w_{34},\beta)=[s+4,s+6]$ and in both cases $s+5$
is a middle color of the sets $S(w_{17},\beta)$ and
$S(w_{34},\beta)$.

Now assume that $\beta(v_{1}w_{14})=s+2$. In this case the colors of
all edges along the cycle
$C=w_{14}v_{1}w_{17}v_{7}w_{37}v_{3}w_{34}v_{4}w_{14}$ are also
known. By the symmetry of the graph $\widehat{K}_{3,4}$, we may
assume that $\beta(v_{1}w_{15})=s+3$ and $\beta(v_{1}w_{16})=s+4$.
Since $\beta(v_{7}w_{27})=s+8$, we have $\min S(v_{2},\beta)\geq
s+3$. On the other hand, since $\beta(v_{4}w_{24})=s+2$, we have
$\min S(v_{2},\beta)\leq s+4$. We consider two cases.

Case 1: $\min S(v_{2},\beta)= s+3$.

In this case $\beta(v_{2}w_{26})\leq s+5$ and thus $\beta(uw_{36})=
s+9$. This implies that $\beta(v_{3}w_{36})=s+7$ and
$\beta(v_{6}w_{36})=s+8$. Since $\beta(v_{1}w_{16})=s+4$, we have
$\beta(v_{6}w_{16})=s+6$ and $\beta(v_{6}w_{26})=s+7$. Also, since
$\beta(v_{2}w_{26})\leq s+5$, we have $\beta(uw_{26})= s+6$. On the
other hand, $\beta(uw_{17})= s+6$, but this is a contradiction,
since $S(w_{17},\beta)=S(w_{26},\beta)=[s+5,s+7]$ and in both cases
$s+6$ is a middle color of the sets $S(w_{17},\beta)$ and
$S(w_{26},\beta)$.

Case 2: $\min S(v_{2},\beta)= s+4$.

In this case $\beta(uw_{24})=s+3$ and $\beta(v_{2}w_{24})= s+4$.
This implies that $\beta(v_{2}w_{25})\geq s+5$ and thus
$\beta(uw_{15})= s+1$. Since $\beta(v_{1}w_{15})=s+3$, we have
$\beta(v_{5}w_{15})=s+2$. Also, since $\beta(v_{3}w_{35})\geq s+6$,
we have $\beta(v_{5}w_{35})= s+4$ and $\beta(v_{5}w_{25})= s+3$.
This implies that $\beta(uw_{25})= s+4$. On the other hand,
$\beta(uw_{34})= s+4$, but this is a contradiction, since
$S(w_{34},\beta)=S(w_{25},\beta)=[s+3,s+5]$ and in both cases $s+4$
is a middle color of the sets $S(w_{34},\beta)$ and
$S(w_{25},\beta)$. ~$\square$
\end{proof}

\bigskip

\section{Interval non-edge-colorable bipartite multigraphs}\

In this section we consider bipartite multigraphs, and first we show
that any bipartite multigraph $G$ with at most four vertices is
interval colorable.

\begin{theorem}
\label{mytheorem3.1} If $G$ is a connected bipartite multigraph with
$\vert V(G)\vert\leq 4$, then $G\in \mathfrak{N}$.
\end{theorem}
\begin{proof}
The cases $\vert V(G)\vert\leq 3$ are trivial. Assume that $\vert
V(G)\vert= 4$. If the underlying graph of $G$ is a tree, then the
proof is trivial, too. Now let $V(G)=\{u,v,w,z\}$ and
$E(G)=E(uv)\cup E(vw)\cup E(wz)\cup E(uz)$ with $\mu(uv)=a$,
$\mu(vw)=b$, $\mu(wz)=c$, $\mu(uz)=d$. Without loss of generality we
may assume that ${\max} \{a,b,c,d\}=d$. We color edges from $E(uv)$
with colors $d+1,\ldots,d+a$, edges from $E(vw)$ with colors
$d-b+1,\ldots,d$, edges from $E(wz)$ with colors $d+1,\ldots,d+c$,
and edges from $E(uz)$ with colors $1,\ldots,d$. If $a<c$, then the
obtained coloring is an interval $(d+c)$-coloring of the multigraph
$G$; otherwise the obtained coloring is an interval $(d+a)$-coloring
of the multigraph $G$. ~$\square$
\end{proof}

\begin{figure}[h]
\begin{center}
\includegraphics[width=15pc]{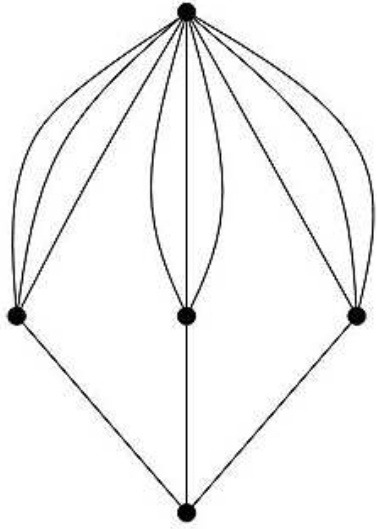}\\
\caption{The bipartite multigraph $G$.}\label{fig7}
\end{center}
\end{figure}

Note that the bipartite multigraph $G$ with $\vert V(G)\vert=5$ and
$\Delta(G)=9$ shown in Fig. \ref{fig7} has no interval coloring. Now
we show a more general result.

Let us define parachute multigraphs $Par(r_{1},\ldots,r_{n})$
($r_{1}\geq\cdots\geq r_{n}\geq 1$) as follows:
\begin{center}
$V(Par(r_{1},\ldots,r_{n}))=\{u,w,v_{1},\ldots,v_{n}\}$,
\end{center}
\begin{center}
$E(Par(r_{1},\ldots,r_{n}))=\{uv_{i}:\mu(uv_{i})=r_{i},1\leq i\leq
n\}\cup \{v_{j}w:1\leq j\leq n\}$.
\end{center}

Clearly, $Par(r_{1},\ldots,r_{n})$ is a connected bipartite
multigraph with $\vert V(Par(r_{1},\ldots,r_{n}))\vert=n+2$, $\Delta
(Par(r_{1},\ldots,r_{n}))=d(u)=\underset{i=1}{\overset{n}{\sum
}}r_{i}$, and $d(w)=n$, $d(v_{i})=r_{i}+1$, $i=1,2,\ldots,n$.

\begin{theorem}
\label{mytheorem3.2} If $\underset{i=3}{\overset{n}{\sum }}r_{i}\geq
n+1$ ($n\geq 3$), then $Par(r_{1},\ldots,r_{n})\notin \mathfrak{N}$.
\end{theorem}
\begin{proof} Suppose, to the contrary, that the multigraph
$Par(r_{1},\ldots,r_{n})$ has an interval $t$-coloring $\alpha$ for
some $t\geq \underset{i=1}{\overset{n}{\sum }}r_{i}$.

Consider the vertex $u$. Let $v_{i_{0}}$ and $v_{i_{1}}$ be two
vertices adjacent to $u$ such that
$\alpha(uv_{i_{0}})={\min}~S(u,\alpha)=s$ and
$\alpha(uv_{i_{1}})={\max}~S(u,\alpha)=s+\underset{i=1}{\overset{n}{\sum
}}r_{i}-1$.

If $i_{0}=i_{1}$, then $\alpha(uv_{i_{0}})={\max
}~S(u,\alpha)=s+\underset{i=1}{\overset{n}{\sum }}r_{i}-1\leq
s+d(v_{i_{0}})-1=s+r_{i_{0}}$ and thus
$\underset{i=1}{\overset{n}{\sum }}r_{i}-1\leq r_{i_{0}}$, which is
impossible.

If $i_{0}\neq i_{1}$, then, by the construction of the multigraph
$Par(r_{1},\ldots,r_{n})$, we have
\begin{center}
$\alpha(v_{i_{0}}w)\leq s+d(v_{i_{0}})-1=s+r_{i_{0}}$ and thus
\end{center}
\begin{center}
$\alpha(v_{i_{1}}w)\leq s+r_{i_{0}}+d(w)-1=s+r_{i_{0}}+n-1$.
\end{center}

This implies that
\begin{center}
$s+\underset{i=1}{\overset{n}{\sum
}}r_{i}-1=\alpha(uv_{i_{1}})={\max }~S(u,\alpha)\leq
s+r_{i_{0}}+n-1+d(v_{i_{1}})-1=s+r_{i_{0}}+r_{i_{1}}+n-1$.
\end{center}

Hence
\begin{center}
$\underset{i=3}{\overset{n}{\sum }}r_{i}\leq
\underset{i=1}{\overset{n}{\sum }}r_{i} - (r_{i_{0}}+r_{i_{1}})\leq
n$,
\end{center}
which is a contradiction.
 ~$\square$
\end{proof}

By Fig. \ref{fig7} and Theorem \ref{mytheorem3.2}, we have

\begin{corollary}
\label{mycorollary3.1} For any positive integer $\Delta \geq 9$,
there is a connected bipartite multigraph $G$ with $G\notin
\mathfrak{N}$ and $\Delta(G)=\Delta$.
\end{corollary}

On the other hand, now we prove that all subcubic bipartite
multigraphs have an interval coloring.

\begin{theorem}
\label{mytheorem3.3} If $G$ is a bipartite multigraph with
$\Delta(G)\leq 3$, then $G\in \mathfrak{N}$ and $w(G)\leq 4$.
\end{theorem}
\begin{proof} First note that if $\Delta(G)\leq 2$, then $G\in \mathfrak{N}$ and $w(G)\leq 2$.

Now suppose that $\Delta(G)=3$. For the proof, it suffices to show
that $G$ has either an interval $3$-coloring or an interval
$4$-coloring.

We show it by induction on $\vert E(G)\vert$. The statement is
trivial for the case $\vert E(G)\vert\leq 4$. Assume that $\vert
E(G)\vert\geq 5$, and the statement is true for all multigraphs
$G^{\prime}$ with $\Delta(G^{\prime})=3$ and $\vert
E(G^{\prime})\vert<\vert E(G)\vert$.

Let us consider a multigraph $G$. Clearly, $G$ is connected. If $G$
has no multiple edges, then the statement follows from Theorem
\ref{mytheorem2}. Now suppose that $G$ has multiple edges.

Let $uv\in E(G)$ and $\mu(uv)\geq 2$. We consider two cases.

Case 1: $\mu(uv)=d_{G}(v)=2$ and $d_{G}(u)=\Delta(G)=3$.

Clearly, in this case there is an edge $uw$, which is a bridge in
$G$. Let us consider a multigraph $G^{\prime}=G-E(uv)$, where
$E(uv)=\{e_{1},e_{2}\}$. By induction hypothesis, $G^{\prime}$ has
either an interval $3$-coloring $\alpha$ or an interval $4$-coloring
$\alpha$.

Subcase 1.1: $\alpha (uw)\leq 2$.

We color the edge $e_{i}$ with color $\alpha (uw)+i$, $i=1,2$. It is
not difficult to see that the obtained coloring is an interval
$3$-coloring or an interval $4$-coloring of the multigraph $G$.

Subcase 1.2: $\alpha (uw)\geq 3$.

We color the edge $e_{i}$ with color $\alpha (uw)-i$, $i=1,2$. It is
not difficult to see that the obtained coloring is an interval
$3$-coloring or an interval $4$-coloring of the multigraph $G$.

Case 2: $\mu(uv)=2$ and $d_{G}(u)=d_{G}(v)=\Delta(G)=3$.

Clearly, in this case there are vertices $x,y$ ($x\neq y$) in $G$
such that $ux\in E(G)$ and $vy\in E(G)$. Let us consider a
multigraph $G^{\prime}=(G-E(uv)-ux-vy)+xy$, where
$E(uv)=\{e_{1},e_{2}\}$. By induction hypothesis, $G^{\prime}$ has
either an interval $3$-coloring $\alpha$ or an interval $4$-coloring
$\alpha$.

Subcase 2.1: $\alpha (xy)\leq 2$.

We delete the edge $xy$ and color the edges $ux$ and $vy$ with color
$\alpha (xy)$ and the edge $e_{i}$ with color $\alpha (xy)+i$,
$i=1,2$. It is not difficult to see that the obtained coloring is an
interval $3$-coloring or an interval $4$-coloring of the multigraph
$G$.

Subcase 2.2: $\alpha (xy)\geq 3$.

We delete the edge $xy$ and color the edges $ux$ and $vy$ with color
$\alpha (xy)$, and the edge $e_{i}$ with color $\alpha (xy)-i$,
$i=1,2$. It is not difficult to see that the obtained coloring is an
interval $3$-coloring or an interval $4$-coloring of the multigraph
$G$.
 ~$\square$
\end{proof}

\bigskip

\section{Problems}\

Finally, we restate the problem posed by Jensen and Toft and formulate a similar problem
for multigraphs. The problems are following:

\begin{problem}
Is there a bipartite graph G with $4\leq \Delta(G)\leq 10$ and
$G\notin \mathfrak{N}$?
\end{problem}

\begin{problem}
Is there a bipartite multigraph G with $4\leq \Delta(G)\leq 8$ and
$G\notin \mathfrak{N}$?
\end{problem}

Since all bipartite graphs of order at most $14$ are interval
colorable \cite{b5} and the bipartite graph $\widehat{K}_{2,2,2}$
with $\vert V(\widehat{K}_{2,2,2})\vert =19$ is not interval
colorable, we would like to suggest the following

\begin{problem}
Is there a bipartite graph G with $15\leq \vert V(G)\vert \leq 18$
and $G\notin \mathfrak{N}$?
\end{problem}

\begin{acknowledgement}
We would like to thank Rafayel R. Kamalian for his attention to this
work. We also would like to thank our referees for many useful
comments and suggestions which helped us to improve the presentation
of the article.
\end{acknowledgement}

\end{document}